\documentclass[journal]{IEEEtran}
\usepackage[utf8]{inputenc}

\title{Exact SDP Formulation for Discrete-Time Covariance Steering with Wasserstein Terminal Cost}
\author{Isin M. Balci  \and \hspace{2cm} Efstathios Bakolas \thanks{This research has been supported  in part by NSF award CMMI-1937957. Isin M. Balci (graduate student) and Efstathios Bakolas (Associate Professor) are with the Department of Aerospace Engineering and Engineering Mechanics at The University of Texas at Austin, Austin, Texas 78712-1221, USA, Email: isinmertbalci@utexas.edu, bakolas@austin.utexas.edu.}}

\usepackage{amsmath}
\usepackage{amssymb}
\usepackage{amsthm}
\usepackage{graphicx}
\usepackage{bm}
\usepackage{subcaption}
\usepackage{tikz}
\usepackage{cite}
\usepackage{pgfplots}
\usepackage{filecontents}
\newcommand{\R}[1]{\mathbb{R}^{#1}}
\newcommand{\tr}[1]{\ensuremath{\mathrm{tr}\left( #1 \right)}}
\newcommand{\E}[1]{\mathbb{E} [ #1 ]}
\newcommand{\Cov}[1]{\mathrm{Cov}( #1 )}
\renewcommand{\t}{^{\mathrm{T}}}
\renewcommand{\S}[1]{\mathbb{S}_{#1}}

\newcommand{\bmu}{\bm{u}}
\newcommand{\bmw}{\bm{w}}
\newcommand{\bmx}{\bm{x}}

\newcommand{\cK}{\bm{\mathcal{K}}}
\newcommand{\bmtheta}{\bm{\Theta}}
\newcommand{\Imat}[1]{\mathbf{I}_{#1}}
\newcommand{\parantheses}[1]{\left( #1 \right) }
\renewcommand{\tr}[1]{\operatorname{tr} \left( #1 \right)}
\newcommand{\bdiag}[1]{\operatorname{bdiag}\left( #1 \right)}

\newcommand{\vertcat}[1]{\operatorname{vertcat} ( #1  )}
\newcommand{\curly}[1]{\{ #1 \}}

\theoremstyle{definition}
\newtheorem{problem}{Problem}

\theoremstyle{plain}
\newtheorem{lemma}{Lemma}
\theoremstyle{definition}
\newtheorem*{remark}{Remark}
\newtheorem{theorem}{Theorem}

\newtheorem{assumption}{Assumption}

\begin{document}

\maketitle

\begin{abstract}
    In this paper, we present new results on the covariance steering problem with Wasserstein distance terminal cost. 
    We show that the state history feedback control policy parametrization, which has been used before to solve this class of problems, requires an unnecessarily large number of variables and can be replaced by a randomized state feedback policy which leads to more tractable problem formulations without any performance loss. 
    In particular, we show that under the latter policy, the problem can be equivalently formulated as a semi-definite program (SDP) which is in sharp contrast with our previous results that could only guarantee that the stochastic optimal control problem can be reduced to a difference of convex functions program. 
    Then, we show that the optimal policy that is found by solving the associated SDP corresponds to a deterministic state feedback policy.
    Finally, we present non-trivial numerical simulations which show the benefits of our proposed randomized state feedback policy derived from the SDP formulation of the problem over existing approaches in the field in terms of computational efficacy and controller performance.
\end{abstract}


\section{Introduction}\label{s:introduction}
In this paper, we consider Covariance Steering (CS) problems for discrete-time linear stochastic systems. 
In this particular class of stochastic optimal control problems, one tries to find a control policy that will steer the probability distribution of the state, or more precisely the first two moments of the latter distribution, to a goal distribution in finite time or infinite time.
Specifically, we address a soft constrained version of the CS problem in which the objective is defined as finding a control policy that will minimize the sum of the expected value of a running cost (e.g., control effort) and a terminal cost which corresponds to a distance metric in the space of probability distributions. 
To this end, we choose the (squared) Wasserstein distance \cite{p:givens1984class}.


\noindent\textit{Literature Review:}
Early versions of the CS problems were focused on the infinite horizon case \cite{p:skelton1987covassignment, p:skelton1987covcontroltheory, p:skelton1992improvedcovariance} for both discrete-time and continuous-time systems.
Finite-horizon CS problems in continuous time are addressed in \cite{p:chen2015covariance1, p:chen2015covariance2, p:chen2018covariance3}. A soft constrained version of the CS in continuous time which is based on the utilization of a Wasserstein distance terminal cost is addressed in \cite{p:halder2016covariancewasserstein}. However, the latter problem formulation does not offer computational advantages over standard CS problem formulations given that for its solution one has to utilize numerical optimal control techniques (e.g., shooting methods) that rely, in general, on nonlinear programming (NLP) methods.
Finite horizon discrete-time CS problems are typically addressed by means of optimization based approaches  \cite{p:bakolas2018covarianceautomatica, p:bakolas2019liouville, p:bakolas2018partialcovariance, p:kotsalis2021robustcovariance, p:goldshtein2017covariance,p:okamoto2018covariancechance, p:balci2021covariancedisturbance, p:okamoto2019pathplanning}.
Furthermore, the multi-agent CS problems are addressed in \cite{p:Saravanos2021DistributedCS}.
And these methods are extended to the nonlinear CS problems in \cite{p:bakolas2020greedynonlinearcovariance, p:ridderhof2019nonlinearcovariance, p:yi2020ddpcovariance}.
However, all of these approaches 
rely on a semi-definite relaxation of the terminal covariance constraint to reduce the CS problem to a convex program.
Recently, a maximum entropy finite-horizon CS problem for discrete-time deterministic linear systems was addressed in \cite{p:ito2022maxentropyCS} based on a Riccati equation approach.



In our previous work \cite{p:balci2020covariancewasserstein}, we have addressed the discrete-time version of the CS problem with Wasserstein terminal cost proposed in \cite{p:halder2016covariancewasserstein} and showed that it can be recast as a difference of convex functions program (DCP) \cite{p:an2005dcprogramming} which can be solved for local optimality with convergence guarantees using the convex-concave procedure (CCP) \cite{p:yuille2003ccp}. 

\noindent\textit{Main Contributions:} As we have already mentioned, finite horizon discrete-time CS problems are typically treated as optimization problems in the relevant literature. The latter optimization problems are either convexified by semi-definite relaxations which prohibit finding the exact solutions to the CS problem or they are cast as generic nonlinear programs, for which there are no guarantees of convergence or global optimality in general.
To the best of our knowledge, this is the first paper that formulates the CS problem with a Wasserstein distance terminal cost as a convex semi-definite program specifically without relying on any convex relaxations of non-convex constraints. 

We first show that the optimal value of the CS problem with state history feedback control policies is lower bounded by the optimal value of the CS problem with the randomized state feedback policy.
Then, we formulate an SDP for solving the CS problem with a Wasserstein distance terminal cost and show that this SDP formulation is exact. 
Then, we express the Wasserstein distance in a convenient form that allows us to express the objective function of the SDP as a linear function of the decision variables.
Next, we show that the optimal policy, whose parameters are obtained by solving the associated SDP, corresponds to a deterministic feedback policy.
Finally, we show the efficacy of our formulation in terms of controller's performance and computation time via numerical simulations.

\noindent\textit{Structure of the Paper:}
The rest of the paper is organized as follows. In Section \ref{s:problem-formulation}, we formulate the CS problem with Wasserstein distance terminal cost. 
We then, summarize previous results based on the state history feedback policy parametrization in Section \ref{s:state-history}.
In Section \ref{s:randomized-policy}, we introduce the randomized state feedback policy parametrization and demonstrate its advantages over other parametrizations. 
In Section \ref{s:SDP-statefeedback}, we formulate the problem as an instance of an SDP using a suitable variable transformation. Numerical simulations are presented in Section \ref{s:numerical-experiments}.
Finally, we conclude the paper in Section \ref{s:conclusion}.

\section{Problem Formulation}\label{s:problem-formulation}
\subsection{Notation}
The space of $n$-dimensional real vectors is denoted as $\R{n}$ and the space of $n\times m$ matrices as $\R{n \times m}$. The space of $n\times n$ symmetric, positive semi-definite and positive definite matrices are denoted by $\S{n}$, $\S{n}^{+}$ and $\S{n}^{++}$, respectively. The $n\times n$ identity matrix is denoted as $\mathbf{I}_{n}$ whereas $\mathbf{0}$ denotes the zero matrix (or vector) with appropriate dimension. 
For $A, B \in \S{n}$, $A \succ B$ ($A \succeq B$) means $A - B \in \S{n}^{++}$ ($A - B \in\S{n}^{+}$).
We use $\tr{\cdot}$ to denote the trace operator. 
$\bdiag{A_1, A_2, \dots, A_{N}}$ denotes the block diagonal matrix whose diagonal blocks are the matrices $A_1, A_2, \dots, A_{N}$. Vertical concatenation of vectors $\{ x_i \in \mathbb{R}^n \}_{i=1}^N$ is denoted as $\vertcat{x_0, \dots, x_{N}}$.
$x \sim \mathcal{N}(\mu, \Sigma)$ means that $x$ is a normal (Gaussian) random variable with mean $\mu\in\mathbb{R}^n$ and covariance matrix $\Sigma \in \S{n}^{+}$. The expectation and the covariance of a random variable $x$ are denoted as $\E{x}$ and $\Cov{x}$, respectively.

\subsection{Wasserstein Distance}
The Wassertein distance defines a distance metric in the space of probability distributions over $\R{n}$. 
In particular, for two random variables over $\R{n}$ with probability density functions $\rho_1$ and $\rho_2$, their squared Wasserstein distance is denoted as $W_2^2(\rho_1, \rho_1)$ and defined as:
\begin{align}\label{eq:Wasserstein-definition}
    W^2_2(\rho_1, \rho_2) := \inf_{\sigma \in \mathcal{P}(\rho_1, \rho_2)} \mathbb{E}_{y} [ \lVert x_1 - x_2 \rVert^2_2]
\end{align}
where $\mathcal{P}(\rho_1, \rho_2)$ denotes the space of all probability distributions of the random variable $y = [x_1\t, x_2\t]\t$ over $\R{2n}$ with finite second moments and marginals $\rho_1$ and $\rho_2$ on $x_1$ and $x_2$, respectively. 

Typically, the Wasserstein distance for two arbitrary probability density functions does not admit an analytic expression and its computation can be a hard task. However, if both $\rho_1$ and $\rho_2$ correspond to the densities of two Gaussian distributions with mean vectors $\mu_1, \mu_2 \in \R{n}$ and covariance matrices $\Sigma_1, \Sigma_2 \in \mathbb{S}^{+}_n$, then the Wasserstein distance admits the following closed form expression \cite{p:givens1984class}:
\begin{align}\label{eq:Wasserstein-Gaussian}
    W_2^2 (\rho_1, \rho_2) := & \lVert \mu_1 - \mu_2 \rVert_2^2 + \tr{\Sigma_1 + \Sigma_2} \nonumber \\ 
    & - 2\tr{ \parantheses{\sqrt{\Sigma_2} \Sigma_1 \sqrt{\Sigma_2}}^{1/2} }.
\end{align}

\subsection{Problem Formulation}
Let us consider the following discrete-time stochastic linear system:
\begin{align}\label{eq:system-dynamics}
    x_{k+1} = A_k x_k + B_k u_k + w_k,
\end{align}
where $\{x_k\}_{k=0}^N$ is the state process over $\R{n}$, $\curly{u_k}_{k=0}^{N-1}$ is the control process over $\R{m}$ and $\curly{w_k}_{0}^{N-1}$ is the disturbance process over $\R{n}$. 
In addition, $A_k \in \R{n \times n}$, $B_k \in \R{n \times m}$ are known matrices for all $k$.
We consider the case where $\curly{w_k}_{k=0}^{N-1}$ is a Gaussian white noise process i.e. $\E{w_k} = \bm{0}$ for all $k \in \curly{0, \dots, N-1}$ and $\E{w_k w_\ell\t} = \bm{0}$ for all $k \neq \ell \in \curly{0, \dots, N-1} $ and $\E{w_k w_k\t} = W_k \in \mathbb{S}_n^{+}$. 
We also assume that the initial state $x_0 \sim \mathcal{N}(\mu_0, \Sigma_0)$ and 
$\E{x_0 w_k\t} = \bm{0}$ for all $k \in \curly{0, \dots, N-1}$.

The main objective in this paper is to find a control policy 
such that the distribution of the terminal state $x_N$ under the policy is close to a desired terminal Gaussian distribution in terms of the Wasserstein metric
while minimizing the expected value of the quadratic control cost.
Next, we state the precise problem formulation.

\begin{problem}\label{problem:first formulation}
Let $\mu_0, \mu_\mathrm{d} \in \R{n}$, $\Sigma_0, \Sigma_d \in \mathbb{S}_n^{++}$, $\curly{R_k}_{k=0}^{N-1}$, $N \in \mathbb{Z}^{+}$, $\lambda > 0$ and $\curly{A_k, B_k, W_k}_{k=0}^{N-1}$ be given where $W_k \in \mathbb{S}_n^{+}$ and $R_k \in \mathbb{S}_m^{++}$ for all $k \in \mathbb{Z}^+$. Furthermore, let $\Pi_0$ denote the space of admissible causal control policies $\curly{m_0(X^0), m_1(X^1), \dots, m_{N-1}(X^{N-1})}$ for system \eqref{eq:system-dynamics} with $u_k = m_k(X^k)$ where $X = \curly{x_0, \dots, x_k}$ is the state sequence up to time step $k$. Then, find a policy $\pi \in \Pi_0$ that solves the following problem:
\begin{subequations}\label{eq:infinite-dim-problem}
\begin{align}
    \min_{\pi \in \Pi_0} & ~~ J(\pi) \label{eq:infinite-dim-problem-objective}\\
    \text{s.t.} & ~~ \eqref{eq:system-dynamics}\quad \forall k \in \{0, \dots, N-1 \} \nonumber \\
    & ~~ u_k = m_k(X^k) \\
    & ~~ x_0 \sim \mathcal{N}(\mu_0, \Sigma_0)
\end{align}
\end{subequations}
where $J(\pi) := \E{\sum_{k=0}^{N-1} u_k\t R_k u_k} + \lambda  W_2^2(\rho_N, \rho_\mathrm{d})$. In addition, $\rho_N$ and $\rho_\mathrm{d}$ denote the probability density functions of the terminal state distribution and the desired distribution, respectively.
\end{problem}

\begin{remark}
Note that Problem \ref{problem:first formulation} is an infinite-dimensional problem because the optimization is over $\Pi_0$ which is the set of all causal policies. 
Thus, the variables of Problem \ref{problem:first formulation} are the control laws $m_k(X^k)$ which are functions of state sequences.
\end{remark}

\section{Solution via State History Feedback Policy}\label{s:state-history}
Since Problem \ref{problem:first formulation} is an infinite-dimensional optimization problem, it is generally computationally intractable. 
To improve computational tractability, 
we restrict the set of admissible policies to the one that are affine functions of states visited up to time step $k$ as follows:
\begin{align}\label{eq:affine-policy}
    m^{(1)}_k(X^k) = v^{(1)}_k + \sum_{\tau = 0}^{k} K^{(1)}_{k, \tau} (x_\tau - \Bar{x}_\tau),
\end{align}
where $\Bar{x}_\tau = \E{x_\tau}$. 
We denote the space of policies that are defined as in \eqref{eq:affine-policy} as $\Pi_1$. 
We observe that all policies $\pi \in \Pi_1$ can be parametrized by $v^{(1)}_k \in \R{m}$ and $K^{(1)}_{k,j} \in \R{m \times n}$ for all $k,j \in \curly{0, \dots, N-1}$ with $k \geq j$.

\subsection{Difference of Convex Functions Program Formulation}
Next, we will show that Problem \ref{problem:first formulation} can be recast as a difference of convex functions program (DCP).
To this aim, we will have to introduce the following variables which will facilitate the subsequent discussion and analysis. 
Let us consider the concatenated vectors $\bmx := \vertcat{x_0, \dots, x_N}$, $\bmu := \vertcat{u_0, \dots, u_{N-1}}$, $\bmw := \vertcat{w_0, \dots, w_{N-1}}$ and $\bm{v} := \vertcat{v_0^{(1)}, \dots, v_{N-1}^{(1)}}$.
Then, the relationship between the concatenated vectors are given as follows:
\begin{subequations}\label{eq:compact-dynamics}
\begin{align}
    \bmx & = \mathbf{\Gamma} x_0 + \mathbf{G_u} \bmu + \mathbf{G_w} \bmw, \\
    \bmu & = \bm{v}  + \cK (\bmx - \bm{\Bar{x}}),  
\end{align}
\end{subequations}
where the matrices $\bm{\Gamma} \in \R{n(N+1) \times n}$, $\mathbf{G_u} \in \R{n(N+1)\times mN}$, $\mathbf{G_w} \in \R{n(N+1) \times nN}$ are derived from \eqref{eq:system-dynamics} and $\cK \in \R{mN \times n(N+1)}$ is derived from \eqref{eq:affine-policy}.
Interested readers can refer to \cite{p:bakolas2018covarianceautomatica} for detailed derivations.

Under the state history feedback policy \eqref{eq:affine-policy}, the mean of the concatenated state vector $\bmx$ is an affine function of $\bm{v}$ as 
$ \bm{\Bar{x}} := \bm{\Gamma} \mu_0 + \mathbf{G_u} \bm{v}.$
To be able to express the covariance of $\bmx$ as a convex quadratic function of the decision variables, we utilize the following bijective variable transformation:
\begin{subequations}\label{eq:variable-transformation1}
\begin{align}
    \bmtheta & = \cK (\Imat{n(N+1)} - \mathbf{G_u} \cK)^{-1}, \\
    \cK & = \bmtheta (\Imat{n(N+1)} - \mathbf{G_u} \bmtheta)^{-1}.\label{eq:variable-transformation1-eq2}
\end{align}
\end{subequations}
Using \eqref{eq:compact-dynamics}-\eqref{eq:variable-transformation1}, and the matrix inversion lemma \cite{b:matrixinversionlemma}, we have
\begin{subequations}\label{eq:covXcovu}
\begin{align}
    \Cov{\bmx} & = (\Imat{n(N+1)} + \mathbf{G_u} \bmtheta) \mathbf{\Bar{S}} \times \nonumber  \\
    & \qquad \qquad (\Imat{n(N+1)} + \mathbf{G_u} \bmtheta)\t, \\
    \Cov{\bmu} & = \bmtheta \mathbf{\Bar{S}} \bmtheta\t .
\end{align}
\end{subequations}
By utilizing \eqref{eq:covXcovu}, one can express the objective function $J(\pi)$
of Problem \ref{problem:first formulation} with admissible policy space $\Pi_1 \subset \Pi_0$ in terms of $(\bm{v}, \bmtheta)$ as follows:
\begin{align}
    \Tilde{J}(\bm{v}, \bmtheta) := J_1(\bm{v})+J_2(\bmtheta) + J_3(\bmtheta) - J_4(\bmtheta)
\end{align}
where $J_i$, for $i\in \{1, \dots, 4\}$, are defined as \cite{p:balci2020covariancewasserstein}:
\begin{subequations}\label{eq:Jdecomp}
\begin{align}
    J_1(\bm{v}) & := \bm{v}\t \bm{\mathcal{R}} \bm{v} + \lambda \lVert \mathbf{F} (\mathbf{\Gamma} \mu_0 + \mathbf{G_u} \bm{v}) - \mu_\mathrm{d} \rVert_2^2 \\
    J_2(\bmtheta) & := \tr{\bm{\mathcal{R}} \bmtheta \mathbf{\Bar{S}} \bmtheta\t} \\
    J_3(\bmtheta) & := \lambda \operatorname{tr} \big( \mathbf{F} (\Imat{n(N+1)} + \mathbf{G_u} \bmtheta)~ \mathbf{\Bar{S}} \times \nonumber \\
    & \qquad \qquad (\Imat{n(N+1)} + \mathbf{G_u} \bmtheta)\t \mathbf{F}\t  + \Sigma_\mathrm{d} \big) \\
    J_4(\bmtheta) & := 2 \lambda \operatorname{tr} \bigg( \big( \sqrt{\Sigma_\mathrm{d}} \mathbf{F} (\Imat{n(N+1)} + \mathbf{G_u} \bmtheta) ~ \mathbf{\Bar{S}} \times \nonumber \\
    & \qquad (\Imat{n(N+1)} + \mathbf{G_u} \bmtheta)\t \mathbf{F}\t \sqrt{\Sigma_\mathrm{d}} \big)^{1/2} \bigg)
\end{align}
\end{subequations}
where $\bm{\mathcal{R}} := \bdiag{R_0, \dots, R_{N-1} } \in \mathbb{S}_{mN}^{++}$ and $\mathbf{F} := [\bm{0}, \dots, \bm{0}, \Imat{n}] \in \R{n \times n (N+1)}$. 
\begin{remark}
Note that in \cite{p:balci2020covariancewasserstein}, we showed that the objective function of Problem \ref{problem:first formulation} when the admissible policy space is $\Pi_1$ can be written as a difference of two convex functions, which does not necessarily mean that it is non-convex. 
However, in our more recent work \cite{p:balci2021convexity-wasserstein}, we showed that the objective function $J(\bm{v}, \bmtheta)$ is in general non-convex and admits more than one local minima.
We also provided a sufficient condition under which the objective function becomes convex, but it is not possible to check whether this condition holds without solving Problem \ref{problem:first formulation} first.
\end{remark}

\section{Randomized State Feedback Policy}\label{s:randomized-policy}
In this section, we define a new set of policies which are comprised of randomized affine state feedback policies, denoted as $\Pi_2$.
Each policy $\pi \in \Pi_2$ is a sequence of feedback laws $\{m^{(2)}_0(x_0), m^{(2)}_1(x_1), \dots, m^{(2)}_{N-1}(x_{N-1}) \}$ where each $m^{(2)}_k$ is a function from $\R{n}$ to the set of $m$ dimensional random variables with Gaussian distribution.
The functions $m^{(2)}_k(x_k)$ is realized as follows:
\begin{align}\label{eq:random-state-policy}
    m^{(2)}_k(x_k) = v^{(2)}_k + K^{(2)}_k (x_k - \Bar{x}_k) + n_k
\end{align}
where $n_k \sim \mathcal{N}(\bm{0}, Q^{(2)}_k)$, $\E{n_k x_k\t} = \bm{0}$ and $Q^{(2)}_k \in \mathbb{S}^{+}_{n}$.
Since every policy $\pi^{(2)} \in \Pi_2$ is parametrized by the terms $\{ \Bar{u}^{(2)}_k, K^{(2)}_k, Q^{(2)}_k \}_{k=0}^{N-1} $, solving Problem \ref{problem:first formulation} over $\Pi_2$ is a finite-dimensional optimization problem. 

Given the discrete-time stochastic linear system \eqref{eq:system-dynamics}, the dynamics of the mean and the covariance of the state process $\curly{x_k}_{k=0}^{N} $ under the randomized state feedback policy defined in \eqref{eq:random-state-policy} are given as follows:
\begin{subequations}\label{eq:state-mean-cov-under-random-policy}
\begin{align}
    \mu_{k+1} & = A_k \mu_k + B_k v_k^{(2)}, \label{eq:state-mean-random-policy}\\
    \Sigma_{k+1} & = (A_k + B_k K^{(2)}_k) \Sigma_k (A_k + B_k K^{(2)}_k)\t \nonumber \\
    & \qquad + B_k Q^{(2)}_k B_k\t + W_k, \label{eq:state-cov-under-random-policy}
\end{align}
\end{subequations}
where $\mu_k$ and $\Sigma_k$ denote the mean vector and the covariance matrix of $x_k$ for all $k \in \{0, \dots, N-1\}$, respectively. In addition, let the mean and covariance of $\curly{u_k}_{k=0}^{N-1}$ be denoted as, respectively, $\Bar{u}_k$ and $U_k$ and written in terms of the policy parameters $\curly{v_k^{(2)}, K_k^{(2)}, Q_k^{(2)}}_{k=0}^{N-1}$ as follows:
\begin{align}\label{eq:control-mean-cov-under-random-policy}
    \Bar{u}_k = v_k^{(2)}, \quad
    U_k = K^{(2)}_k \Sigma_k K^{(2)\mathrm{T}}_k + Q^{(2)}_k.
\end{align}

Moreover, the dynamics of the state mean and the state covariance under the state history feedback policy defined in \eqref{eq:affine-policy} are given by:
\begin{subequations}
\begin{align}
    \mu_{k+1} & = A_k \mu_k + B_k v_k^{(1)} \\
    \Sigma_{k+1} & = A_k \Sigma_{k} A_k\t + A_k L_k\t B_k\t  + B_k L_k A_k\t  \nonumber \\
    & \quad + B_k U_k B_k\t + W_k \label{eq:state-cov-under-history}
\end{align}
\end{subequations}
where 
\begin{subequations}\label{eq:LandMdef}
\begin{align}
    L_k & = \E{\Tilde{u}_k \Tilde{x}_k\t} = \E{\sum_{\tau=0}^k K^{(1)}_\tau \Tilde{x}_\tau \Tilde{x}_k\t}, \\
    U_k & = \E{\Tilde{u}_k \Tilde{u}_k\t} = \E{ \sum_{\tau, \ell=0}^{k}  K^{(1)}_\tau \Tilde{x}_\tau \Tilde{x}_\ell\t K^{(1)\mathrm{T}}_\ell}.
\end{align}
\end{subequations}
for all $k \in \{0, \dots, N-1\}$. It is worth noting that the term $L_k $ corresponds to the cross-covariance between the random variables $x_k$ and $u_k$.

\begin{assumption}\label{assumption:non-singular-sigma}
For any policy $\pi \in \Pi_1$, the state covariance $\Sigma_k \in \mathbb{S}_n^{++}$
for all $k \in \curly{0, \dots, N}$.
\end{assumption}

We now show that for any policy $\pi \in \Pi_1$, there exists a policy $\Bar{\pi} \in \Pi_2$ such that the first and the second moments of the control and the state processes that are induced by $\Bar{\pi}$ are equal to the ones that are induced by $\pi$ under Assumption \ref{assumption:non-singular-sigma}. 
This will allow us to conclude that the optimal value of Problem \ref{problem:first formulation} when the admissible policy space is $\Pi_1\subset \Pi_0$ is lower bounded by the optimal value of the same problem when the admissible policy space is $\Pi_2\subset \Pi_0$.
The next lemma and theorem formally states the previous claims about the policy spaces $\Pi_1$ and $\Pi_2$. 


\begin{lemma}\label{lemma:equivalence-policies}
Let Assumption \ref{assumption:non-singular-sigma} hold and let $\curly{\mu^{(1)}_k, \Sigma^{(1)}_k}_{k=0}^N$, $\curly{ \Bar{u}_k^{(1)}, U_k^{(1)} }_{k=0}^{N-1}$ be the mean and covariance  of the state and the control sequences under a policy $\pi \in \Pi_1$.
Then, for all $\pi \in \Pi_1$, there exists a policy $\Bar{\pi} \in \Pi_2$ such that 
\begin{subequations}\label{eq:equivalence-relations}
\begin{align}
    \mu_k^{(1)} &= \mu_k^{(2)}, \quad &\Bar{u}_k^{(1)}  &= \Bar{u}_k^{(2)}, \label{eq:mean-equivalence}\\
    \Sigma_k^{(1)} &= \Sigma_k^{(2)}, \quad & U_k^{(1)}&= U_k^{(2)}, \label{eq:covariance-equivalence}
\end{align}
\end{subequations}
for all $k \in \curly{0, \dots, N}$ where $\curly{\mu^{(2)}_k, \Sigma^{(2)}_k}_{k=0}^N$ and  $\curly{\Bar{u}_k^{(2)}, U_k^{(2)}}_{k=0}^{N-1}$ denote the mean and covariance of the state and the control sequences under a policy $\Bar{\pi} \in \Pi_2$.
\end{lemma}

\begin{proof}
Let us consider a policy $\pi \in \Pi_1$ and let $\{v^{(1)}_k, K^{(1)}_{k, \ell} \}_{k=0, \ell=0}^{N-1, k}$ be the parameters corresponding to this policy.
Our goal is to find a different policy $\Bar{\pi} \in \Pi_2$
with corresponding policy parameters $\{ v_k^{(2)}, K_k^{(2)}, Q_k^{(2)} \}_{k=0}^{N-1}$ such that the mean and covariance of the state and control processes will satisfy \eqref{eq:equivalence-relations}. 
First, let $v_k^{(2)} = v_k^{(1)}$. Because the feedback terms do not affect the dynamics of the mean of the state process $\curly{x_k}$ for both sets of policies, \eqref{eq:mean-equivalence} follows immediately.

Now, let $\Sigma^{(1)}_k, U^{(1)}_k$ be the covariance matrices of $x_k$ and $u_k$ induced by the state history feedback policy defined in \eqref{eq:affine-policy}. Then, let 
\begin{subequations}\label{eq:KandQdef-lemma-proof}
\begin{align}
    K_k^{(2)} &= L_k (\Sigma^{(1)}_k)^{-1}, \\
    Q_k^{(2)} &= U^{(1)}_k - L_k (\Sigma^{(1)}_k)^{-1} L_k\t,
\end{align}
\end{subequations}
where $L_k$ and $M_k$ are defined as in \eqref{eq:LandMdef}. 
Using \eqref{eq:state-cov-under-random-policy} with the parameters $\curly{K^{(2)}_k, Q^{(2)}_k}$, we can write the state covariance dynamics under policy $\Bar{\pi} \in \Pi_2$, which is denoted by $\Sigma^{(2)}_k$, as follows:
\begin{align}\label{eq:TildeSigmaDynamics}
    \Sigma^{(2)}_{k+1} & = A_k \Sigma^{(2)}_k A_k\t + A_k \Sigma^{(2)}_k (\Sigma_k^{(1)})^{-1} L_k\t B_k\t \nonumber \\
    & + B_k L_k (\Sigma_k^{(1)})^{-1} \Sigma^{(2)}_k  A_k\t \nonumber\\
    & + B_k (L_k (\Sigma_k^{(1)})^{-1} \Sigma^{(2)}_k (\Sigma_k^{(1)})^{-1} L_k\t ) B_k\t \nonumber \\
    & + B_k (U^{(1)}_k - L_k (\Sigma_k^{(1)})^{-1}  L_k\t ) B_k\t + W_k.
\end{align}
The initial state covariance is equal for both policies, that is, $\Sigma_0 = \Sigma^{(1)}_0 = \Sigma^{(2)}_0$.
Furthermore, if we assume that $\Sigma_k^{(1)} = \Sigma_k^{(2)}$ as the inductive hypothesis, then it follows that $\Sigma_{k+1}^{(1)} = \Sigma_{k+1}^{(2)}$ by virtue of \eqref{eq:state-cov-under-history} and \eqref{eq:TildeSigmaDynamics}. Thus, we conclude that $\Sigma_k^{(1)} = \Sigma_k^{(2)}$ holds for all $k$ by induction. 
Using \eqref{eq:control-mean-cov-under-random-policy}, the covariance matrix $U_k^{(2)}$ of the control input $u_k$ under $\curly{K_k^{(2)}, Q_k^{(2)}}$ is given by
\begin{align}\label{eq:TildeSigmaUequation}
    U_k^{(2)} &= U^{(1)}_k - L_k (\Sigma^{(1)}_k)^{-1} L_k\t + L_k (\Sigma^{(1)}_k)^{-1} L_k\t \nonumber \\
    & = U_k^{(1)}
\end{align}
which completes the proof.
\end{proof}

\begin{remark}
Note that Lemma \ref{lemma:equivalence-policies} does not only apply to state history feedback policies \eqref{eq:affine-policy}. 
For any linear or nonlinear control policy $\pi \in \Pi_0$ which induces a state process with non-singular covariance matrix for all time steps, there exists a policy $\Tilde{\pi} \in \Pi_2$ which induces state and control processes which have the same first and second moments as the induced state and control processes under policy $\pi \in \Pi_0$.
\end{remark}

\begin{theorem}
Let $J^{\star}_1 = J(\pi_1^\star)$, where $\pi_1^\star \in \Pi_1$ is the minimizer of Problem \ref{problem:first formulation} over $\Pi_1$, and $J_2^\star = J(\pi_2^\star)$, where $\pi_2^\star$ is the minimizer of Problem \ref{problem:first formulation} over $\Pi_2$.
Then, $J^{\star}_2 \leq J_1^\star$. 
\end{theorem}
\begin{proof}
Let $\pi_1^\star \in \Pi_1$ be the optimal policy that solves Problem~1 over $\Pi_1$ and let $\{v^\star_k, K^\star_{k,\ell} \}_{k=0, \ell=0}^{N-1, k}$ be the corresponding set of parameters. 
Now, let $\curly{v'_k, K'_{k}, Q'_k}_{k=0}^{N-1}$ be the policy parameters of $\pi_2' \in \Pi_2$ which are defined as in \eqref{eq:KandQdef-lemma-proof}.
Then, from Lemma \ref{lemma:equivalence-policies}, we can show that the first two moments of the control and state sequences under policy $\pi_1^\star$ are equal to the ones under policy $\pi_2'$. 
Since both policies are affine functions of the state $x_k$, then the distribution of $x_k$ will be Gaussian at each time step $k$, and the distribution of the terminal state $x_N$ under both policies will also be Gaussian with the same mean  and covariance. Thus, the Wasserstein distance terms in the objective function of Problem \ref{problem:first formulation} are equal for both policies.
Furthermore, the running cost term at time step $k$ is given by:
\begin{align}
    \E{u_k\t R_k u_k} & =  \tr{R_k \E{u_k u_k\t} } \nonumber \\
    & = \Bar{u}_k\t R_k \Bar{u}_k + \tr{R_k U_k}. \label{eq:running-costk}
\end{align}
Thus, the running cost is a function of the first two moments of $u_k$. 
Since the values of $\Bar{u}_k, U_k$ are equal under both policies $\pi'$ and $\pi^\star$, the running cost terms in \eqref{eq:running-costk} are equal at all time steps for both policies.
Thus, we conclude that $J(\pi_2^\star) \leq J(\pi_2') = J(\pi_1^\star)$. 
\end{proof}

\section{SDP formulation for Randomized State Feedback Policy}\label{s:SDP-statefeedback}
In Section \ref{s:randomized-policy}, we showed that the optimal value of Problem \ref{problem:first formulation} over randomized state feedback policies defined in \eqref{eq:random-state-policy} is upper bounded by the optimal value of the same problem over state history feedback policies defined in \eqref{eq:affine-policy}. 
Furthermore, the number of parameters that is required to parametrize the control policies in $\Pi_2$ is less than $\Pi_1$. 
However, Problem \ref{problem:first formulation} can be recast as a difference of convex functions program using the policy class $\Pi_1$ as shown in Section \ref{s:state-history}. 

In this section, we will formulate Problem \ref{problem:first formulation} with admissible policy space $\Pi_2\subset \Pi$ as a convex semi-definite program using a suitable variable transformation. 
In particular, the finite-dimensional problem can be formulated as follows from \eqref{eq:state-mean-cov-under-random-policy}, \eqref{eq:control-mean-cov-under-random-policy}:
\begin{align}\label{eq:problem-NLP-beforeSDP}
    \min_{\substack{v_k, \mu_k, \\ K_k, Q_k, \Sigma_k}} & J_{\mathrm{NLP}}(\curly{v_k, K_k, Q_k}_{k=0}^{N-1})   \\ 
    \text{s.t.} & ~~ \eqref{eq:state-mean-random-policy}, \eqref{eq:state-cov-under-random-policy} \quad \forall k \in \{0, \dots, N-1\} \nonumber
\end{align}
where 
\begin{align}\label{eq:objective-NLP-beforeSDP}
    J_{\mathrm{NLP}} & :=  \sum_{k=0}^{N-1} \Big(  v_k\t R_k v_k + \tr{R_k (K_k \Sigma_k K_k + Q_k) } \Big) \nonumber \\
    & ~~~~ + \lambda \Big(\lVert \mu_N - \mu_\mathrm{d} \rVert_2^2 + \tr{\Sigma_N + \Sigma_\mathrm{d}} \nonumber \\
    & ~~~~~ \qquad -2 \mathrm{tr} \big( (\sqrt{\Sigma_\mathrm{d}} \Sigma_N \sqrt{\Sigma_\mathrm{d}})^{1/2} \big) \Big) 
\end{align}
since $\E{u_k\t R_k u_k} = \E{u_k}\t R_k \E{u_k} + \tr{R_k U_k} $.
The sum of the first two terms of the objective function in \eqref{eq:objective-NLP-beforeSDP} corresponds to the running cost whereas the third term corresponds to the terminal cost which is taken to be the squared Wasserstein distance between the terminal state density and the desired density. The constraints \eqref{eq:state-mean-random-policy} and \eqref{eq:state-cov-under-random-policy} impose the mean and the covariance dynamics given in respective equations.
The optimization problem given in \eqref{eq:problem-NLP-beforeSDP} is a nonlinear program (NLP) which is generally non-convex due to the presence of the terms $\tr{R_k K_k \Sigma_k K_k\t }$ in the running cost and the nonlinear equality constraints \eqref{eq:state-cov-under-random-policy}. 
In order to convexify the latter problem, we first introduce the following variable transformation which is often used in convex optimization based approaches for optimal control and robust control problems for discrete-time linear systems \cite{p:kothare1996robustMPC-LMI, p:chen2015covariance2, p:benedikter2022covariance}:
\begin{align}\label{eq:Pkdef}
    P_k = K_k \Sigma_k.
\end{align}
By introducing the new variable $P_k$ in \eqref{eq:Pkdef}, we can eliminate the bilinear terms in \eqref{eq:state-cov-under-random-policy}. 
Now, observe that the term $K_k \Sigma_k K_k\t$ that appears in \eqref{eq:objective-NLP-beforeSDP} and \eqref{eq:state-cov-under-random-policy} is equal to $P_k \Sigma_k^{-1} P_k\t$ under $\eqref{eq:Pkdef}$.
With this observation, we define a new decision variable $M_k$ as follows:
\begin{align}\label{eq:MKdef}
    M_k = P_k \Sigma_k^{-1} P_k\t + Q_k.
\end{align}
By using $M_k$ and $P_k$, the nonlinear equality constraint \eqref{eq:state-cov-under-random-policy} can be written equivalently as follows:
\begin{align}
    \Sigma_{k+1} & = A_k \Sigma_k A_k\t + A_k P_k\t B_k\t + B_k P_k A_k\t \nonumber \\
    & \qquad + B_k M_k B_k\t + W_k. \label{eq:nonlinear-equality-partA}
\end{align}
The following lemmas will be used to formulate Problem \eqref{eq:problem-NLP-beforeSDP} as a standard SDP.

\begin{lemma}\label{lemma:equivalent-lmi}
Let $X \in \mathbb{S}_n^{+}$, $Y \in \mathbb{S}_m^{++}$ and $N \in \R{n \times m} $. Then, there exists a matrix $Q \in \mathbb{S}_n^{+}$ such that 
\begin{align}\label{eq:lemma-lmi-eq1}
    X = N Y^{-1} N\t + Q
\end{align}
if and only if 
\begin{align}\label{eq:lemma-lmi-eq2}
    \left[ 
    \begin{array}{cc}
    X & N \\
    N\t & Y
    \end{array} 
    \right] \succeq \bm{0}.
\end{align}
\end{lemma}
\noindent The proof of Lemma \ref{lemma:equivalent-lmi} is straight-forward and therefore is omitted.
The equivalence of \eqref{eq:lemma-lmi-eq1} and \eqref{eq:lemma-lmi-eq2} can be seen by applying Schur's complement formula to \eqref{eq:lemma-lmi-eq2}.
\begin{lemma}\label{lemma:equivalent-varphi}
Let $\varphi : \mathbb{S}_n^{+} \times \mathbb{S}_n^+ \rightarrow \R{}$ be defined as:
\begin{align}\label{eq:varphi-definition}
    \varphi(M, N) := - \tr{(\sqrt{M} N \sqrt{M})^{1/2}}.
\end{align}
Then, $\varphi(M, N)$ can be equivalently written as:
\begin{align}
    \varphi(M, N) & = \min_{L \in \mathcal{M}(N, M)} ~~ - \tr{L}, \label{eq:equivalent-varphi-constr}
\end{align}
where the set $\mathcal{M}(N, M)$ is defined as follows:
\begin{align}
\mathcal{M}_{}= \bigg\{ L \in \mathbb{S}^{+}_n ~|~ \left[ \begin{array}{cc}
         N & M^{-1/2} L  \\
         L M^{-1/2} & \Imat{n} 
    \end{array} \right] \succeq \bm{0} \bigg\}.
\end{align}
\end{lemma}
\begin{proof}
Let us consider the following optimization problem:
\begin{subequations}\label{problem:proof-equivalent-varphi}
\begin{align}
    \min_{L \in \mathbb{S}_n^+} & - \tr{L} \\
    \text{s.t.}~~ & (\sqrt{M} N \sqrt{M})^{1/2} \succeq L \label{eq:proof-equivalent-varphi-constr}
\end{align}
\end{subequations}
It follows readily that the minimizer $L^\star$ of the latter problem is equal to $(\sqrt{M} N \sqrt{M})^{1/2}$. Thus, the optimal value of the objective function of the optimization problem in \eqref{problem:proof-equivalent-varphi} is equal to the value of $\varphi(M,N)$ which is defined in \eqref{eq:varphi-definition}. Therefore,
\begin{subequations}
\begin{align}
    \eqref{eq:proof-equivalent-varphi-constr} & \Leftrightarrow \sqrt{M} N \sqrt{M} \succeq L^2 \label{eq:relation-proof1}\\
    & \Leftrightarrow N \succeq M^{-1/2} L \Imat{n} L M^{-1/2} \label{eq:relation-proof2}\\
    & \Leftrightarrow \left[ \begin{array}{cc}
         N & M^{-1/2} L  \\
         L M^{-1/2} & \Imat{n} 
    \end{array} \right] \succeq \bm{0} \label{eq:relation-proof3}
\end{align}
\end{subequations}
The relation in \eqref{eq:relation-proof1} holds since the function $f(X) = X^{2}$ is monotonically non-decreasing for all $X \in \mathbb{S}_n^+$ in the L\"{o}wner partial order sense. 
The relation in \eqref{eq:relation-proof2} holds due to the congruence transform with $M^{-1/2}$.
The final step of the proof in \eqref{eq:relation-proof3} follows from direct application of Schur's complement formula to \eqref{eq:relation-proof2}.
\end{proof}
Using Lemma \ref{lemma:equivalent-lmi} and Lemma  \ref{lemma:equivalent-varphi}, we can write the problem in \eqref{eq:problem-NLP-beforeSDP} as a standard SDP as follows:

\begin{subequations}\label{problem:sdp-formulation}
\begin{align}
    \min_{\substack{v_k, \mu_k, L \\ P_k, \Sigma_k, M_k}} &  J_{\mathrm{SDP}}(\curly{v_k, P_k,  M_k}_{k=0}^{N-1}, L)\\
    \text{s.t.}~~ &  \eqref{eq:state-mean-random-policy}, \eqref{eq:nonlinear-equality-partA} \nonumber\\
    & \left[ \begin{array}{cc}
         M_k & P_k   \\
         P_k\t & \Sigma_k 
    \end{array} \right] \succeq \bm{0} \label{eq:SDP-Mk-constr} \\
    & \left[ \begin{array}{cc}
        \Sigma_N & \Sigma_\mathrm{d}^{-1/2} L   \\
        L \Sigma_\mathrm{d}^{-1/2}& \Imat{n}
    \end{array} \right] \succeq \bm{0} \\
    & ~ L, \Sigma_k, M_k \succeq \bm{0}  \label{eq:SDP-PSD-constr}
\end{align}
\end{subequations}
where the constraints \eqref{eq:state-mean-random-policy}, \eqref{eq:nonlinear-equality-partA}, \eqref{eq:SDP-Mk-constr} and \eqref{eq:SDP-PSD-constr} are imposed for all $k \in \{0, \dots, N-1 \}$. 
In addition, the objective function $J_{\mathrm{SDP}}$ is defined as follows:
\begin{align}
    J_{\mathrm{SDP}} := & \sum_{k=0}^{N-1} v_k\t R_k v_k + \tr{R_k M_k}  \nonumber \\
    & \hspace{-0.50cm} + \lambda \big(\lVert \mu_N - \mu_\mathrm{d} \rVert_2^2 + \tr{ \Sigma_N + \Sigma_\mathrm{d} - 2 L } \big) & 
\end{align}

\begin{theorem}\label{theorem:equivalenceofSDPandNLP}
Let $\curly{\mu'_k, v'_k, \Sigma'_k, K'_k, Q'_k}$ be the minimizer of the NLP in \eqref{eq:problem-NLP-beforeSDP}, and let $\{\mu''_k, v''_k, \Sigma''_k, P''_k, M''_k, L \}$ be the minimizer of the SDP in \eqref{problem:sdp-formulation}. Then,
$v'_k = v''_k$, $\mu'_k = \mu''_k$, $\Sigma'_k = \Sigma''_k, K'_k = P''_k (\Sigma''_k)^{-1}, Q'_k = M''_k - P''_k (\Sigma''_k)^{-1} P^{\prime \prime \mathrm{T}}_k$, for all $k \in \{0, \dots, N-1\}$.
\end{theorem}
\begin{proof}
The equalities $\mu'_k = \mu''_k$ and $v'_k = v''_k$ follow immediately given that the problem in \eqref{problem:sdp-formulation} and the problem in \eqref{eq:problem-NLP-beforeSDP} share the constraint in \eqref{eq:state-mean-random-policy} that enforces the state mean dynamics and in addition, the terms related to the mean dynamics in the objective functions $J_{\mathrm{NLP}}$ and $J_{\mathrm{SDP}}$ are equal. By the definition of $M_k$ in \eqref{eq:MKdef}, it follows that $\tr{R_k M_k } = \tr{R_k (K_k \Sigma_k K_k\t + Q_k)}$. By utilizing the definition of $M_k$, the state covariance dynamics in \eqref{eq:state-cov-under-random-policy} can be equivalently expressed as in \eqref{eq:MKdef} and \eqref{eq:nonlinear-equality-partA}. 
Furthermore, we observe that \eqref{eq:SDP-Mk-constr} is equivalent to \eqref{eq:MKdef} from Lemma \ref{lemma:equivalent-lmi}. 
Finally, the last term of the objective function of the NLP in \eqref{eq:objective-NLP-beforeSDP} can be expressed as the optimal value of a relevant minimization problem in view of \eqref{lemma:equivalent-varphi}. 
Hence, the SDP problem in \eqref{problem:sdp-formulation} is equivalent to the NLP problem in \eqref{eq:problem-NLP-beforeSDP}.
\end{proof}

\begin{remark}\label{remark:separable-problems}
The problem given in \eqref{problem:sdp-formulation} can be decomposed into two optimization problems. 
One for steering the mean and one for steering the covariance.
This is because the objective function $J_{\mathrm{SDP}}$ is equal to the sum of $J_{\mathrm{SDP},\mathrm{mean}}$ and $J_{\mathrm{SDP}, \mathrm{cov}}$, where $J_{\mathrm{SDP},\mathrm{mean}} = \sum_{k=0}^{N-1} v_k^{\mathrm{T}} R_k v_k + \lambda \lVert \mu_N - \mu_{\mathrm{d}} \rVert_2^2$ and $\curly{v_k}_{k=0}^{N-1}$ and $J_{\mathrm{SDP}, \mathrm{cov}} = \sum_{k=0}^{N-1} \tr{R_k M_k} + \lambda \tr{\Sigma_N + \Sigma_d - 2 L}$.
The mean steering problem can be formulated as an equality constrained convex quadratic program as follows:
\begin{align*}
    \min_{v_k, \mu_k} ~& \sum_{k=0}^{N-1} v_k^{\mathrm{T}} R_k v_k + \lambda \lVert \mu_N - \mu_{\mathrm{d}} \rVert_2^2 \\
    \text{s.t.} ~ & \eqref{eq:state-mean-random-policy}
\end{align*}
\end{remark}

Although the proposed control policy~\eqref{eq:random-state-policy} for the solution to the CS problem corresponds to a randomized state feedback policy, it turns out that the optimal policy for the Problem in \eqref{problem:sdp-formulation} corresponds to a deterministic policy, i.e., $Q_k = \bm{0}$ for all $k$. 
This observation is formally stated in the following theorem.
\begin{theorem}
Let us assume that $A_k$ is non-singular for all $k \in \{0,\dots, N-1\}$ and let $\pi' \in \Pi_2$ be the optimal policy that solves the optimization problem in \eqref{eq:problem-NLP-beforeSDP} and let 
$\curly{v_k, K_k, Q_k}_{k=0}^{N-1}$ be its corresponding parameters. Then, $Q_k = \bm{0}$ for all $k \in \curly{0, \dots, N-1}$.
\end{theorem}
\begin{proof}
Let us assume, for the sake of contradiction, that $Q_\ell \neq \bm{0}$ for some $\ell \in \curly{0, \dots, N-1}$. 
Let $\curly{\Sigma_k}_{k=0}^{N}$ be the induced state covariance sequence under the optimal policy $\pi'$, and $M_\ell, P_\ell$ be defined as in \eqref{eq:Pkdef}, \eqref{eq:MKdef}, respectively. 
Let us now consider the following optimization problem:
\begin{subequations}\label{eq:theorem3-problem}
\begin{align}
    \min_{M_\ell, P_\ell} & ~ \tr{R_\ell M_\ell } \\
    \text{s.t.} & ~ \Sigma_{\ell+1} = A_\ell \Sigma_\ell A_\ell\t + A_\ell P_\ell\t B_\ell\t \nonumber \\
    & \qquad ~~~~ + B_\ell P_\ell A_\ell\t + B_\ell M_\ell B_\ell\t + W_\ell \\
    & ~ \begin{bmatrix} 
    M_\ell & P_\ell \\
    P_\ell\t & \Sigma_\ell
    \end{bmatrix} \succeq \bm{0}
\end{align}
\end{subequations}
which can be written in a more compact way after dropping the subscripts as follows:
\begin{subequations}\label{eq:theorem3-problem-compact}
\begin{align}
    \min_{M, P} & ~ \tr{R M} \\
    \text{s.t.} & ~ Z = \Sigma + P\t Y\t + Y P + Y M Y\t \\
    & \begin{bmatrix}
    M & P \\
    P\t & \Sigma
    \end{bmatrix} \succeq \bm{0}
\end{align}
\end{subequations}
where $Z = A_\ell^{-1}(\Sigma_{\ell+1} - W_\ell)A_\ell^{-\mathrm{T}}$ and $Y = A_\ell^{-1} B_\ell$.
Furthermore, the dual problem associated with the problem in  \eqref{eq:theorem3-problem-compact} is given by
\begin{subequations}
\begin{align}
    \max_{H \in \mathbb{S}_n, Q \in \mathbb{S}_n^{+}} & ~ \tr{H(Z-\Sigma) - Q_{22}\Sigma} \label{eq:theorem3-dual-objective}\\
    \text{s.t.} & ~ Q_{11} - (R-Y H Y\t) = \bm{0} \label{eq:theorem3-dual-constr1} \\
    & ~  Q_{12}\t = -HY \label{eq:theorem3-dual-constr2} 
\end{align}
\end{subequations}
where $Q = \left [ \begin{smallmatrix} Q_{11} & Q_{12} \\ Q_{12}\t & Q_{22} \end{smallmatrix} \right]$.
From complementary slackness, we conclude that the optimal values of the primal and dual variables satisfy the following equality:
\begin{align}\label{eq:comp-slack-first}
    Q 
    \begin{bmatrix}
    \Imat{m} & P \Sigma^{-1} \\
    \bm{0} & \Imat{n}
    \end{bmatrix}
    \begin{bmatrix}
    M - P \Sigma^{-1} P\t & \bm{0} \\
    \bm{0} & \Sigma
    \end{bmatrix}= \bm{0}
\end{align}
since $\Sigma$ is non-singular in view of Assumption~\ref{assumption:non-singular-sigma}.
From \eqref{eq:comp-slack-first}, we obtain:
\begin{subequations}\label{eq:comp-slack-eqns}
\begin{align}
    Q_{11} (M - P \Sigma^{-1} P\t) = \bm{0} \label{eq:comp-slack1}\\
    Q_{12}\t (M - P \Sigma^{-1} P\t) = \bm{0} \label{eq:comp-slack3}
\end{align}
\end{subequations}
Now, let $\mathcal{M} = M - P \Sigma^{-1} P\t$. By multiplying both sides of dual feasibility constraint \eqref{eq:theorem3-dual-constr1} by $\mathcal{M}$ on the left, we obtain:
\begin{subequations}
\begin{align}
    \mathcal{M}(Q_{11} - (R - YHY\t)) & = \bm{0}  \\
    \mathcal{M}Q_{11} + \mathcal{M}YHY\t & = \mathcal{M}R,
\end{align}
\end{subequations}
where in our derivations we have used that $\mathcal{M} Q_{11} = \bm{0}$ from \eqref{eq:comp-slack1} and $\mathcal{M} Y H Y\t = \bm{0}$ from \eqref{eq:theorem3-dual-constr2} and \eqref{eq:comp-slack3}. We obtain $\mathcal{M}R=\bm{0}$. Since $R \succ \bm{0}$, we have $\mathcal{M} = \bm{0}$. 
Thus, we have $M_\ell - P_\ell \Sigma_\ell P_\ell\t = Q_\ell = \bm{0}$.
This contradicts our initial assumption that there exists $\ell \in \curly{0, \dots, N-1}$ such that $Q_\ell \neq \bm{0}$ and thus, we conclude that $Q_k = \bm{0}$ for all $k$.
This completes the proof.
\end{proof}



\section{Numerical Experiments}\label{s:numerical-experiments}
We solve the SDP in \eqref{problem:sdp-formulation} for the linear system \eqref{eq:system-dynamics} with parameters: 
   $A_k = \left[\begin{smallmatrix} 1.0 & 0.175 \\ -0.343 & 1.0 \end{smallmatrix} \right]$, 
    $B_k = \left[\begin{smallmatrix} 0.789 \\ 0.415 \end{smallmatrix}\right]$, 
$W_k = 0.5 \Imat{2}$, $\mu_0 = [4.0, 3.0]\t$, $\Sigma_0 = \Imat{2}$, $\mu_\mathrm{d} = [0, 0]\t$, $\Sigma_\mathrm{d} = \left[ \begin{smallmatrix} 3 & -2 \\ -2 & 3 \end{smallmatrix} \right]$, $R_k = 1.0$, $N =60$ to generate the experiment results in Fig. \ref{fig:3d-state-evolution}
and Fig. \ref{fig:comparison-LMIwithPrevious}.
We also consider a two-dimensional and a three-dimensional double integrator system with $A_k = \left[\begin{smallmatrix}\Imat{n} & \Delta t \Imat{n}\\ \bm{0} & \Imat{n}\end{smallmatrix}\right]$, $B_k = \left[ \begin{smallmatrix} \bm{0} \\ \Imat{n} \end{smallmatrix} \right]$
with $\Delta t = 0.05$ and $n = 2,3$ to show that our SDP formulation in \eqref{problem:sdp-formulation} can handle CS problems with higher dimensional systems.
The results of these experiments are given in Table \ref{tab:comp_LMI}. 
In all of our experiments, we use MOSEK \cite{mosek} to solve SDPs along with the modeling tool CVXPY \cite{p:diamond2016cvxpy}.


\begin{figure*}
\centering
\begin{subfigure}{0.32\linewidth}
\includegraphics[height=\linewidth, width=\linewidth]{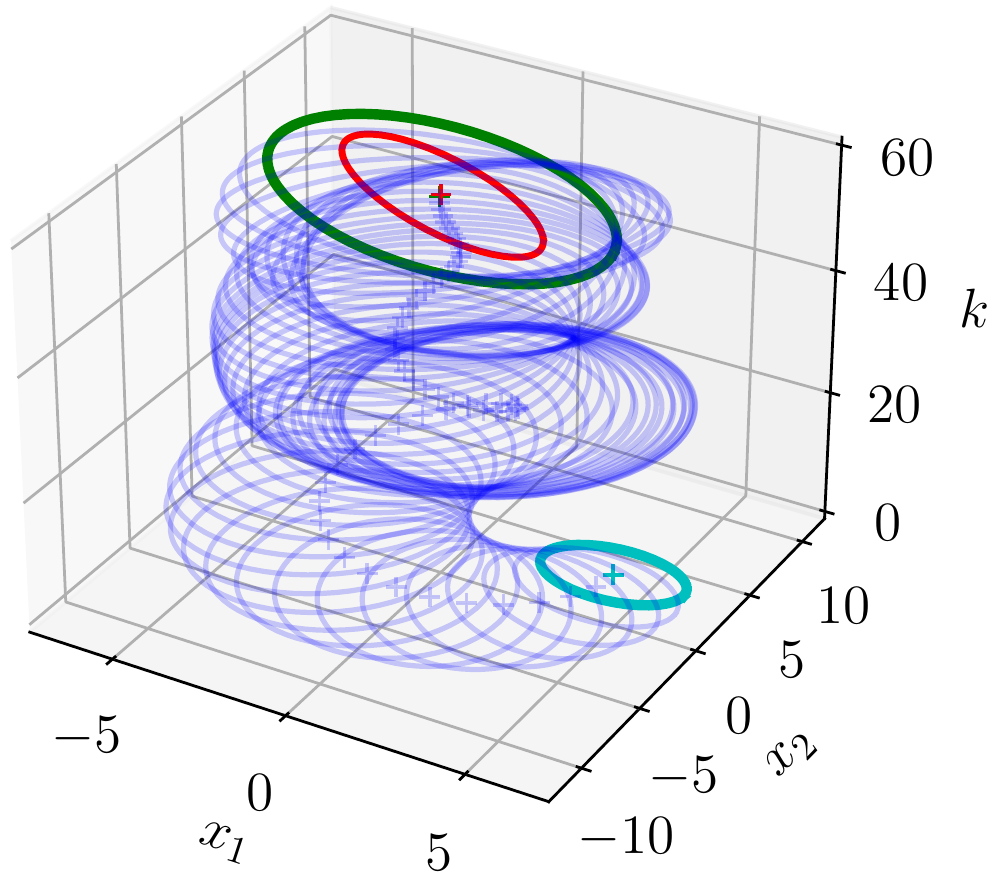}
\caption{$\lambda = 0.5$}\label{subfig:state-evolution-lambda05}
\end{subfigure}
\begin{subfigure}{0.32\linewidth}
\includegraphics[height=\linewidth,width=\linewidth]{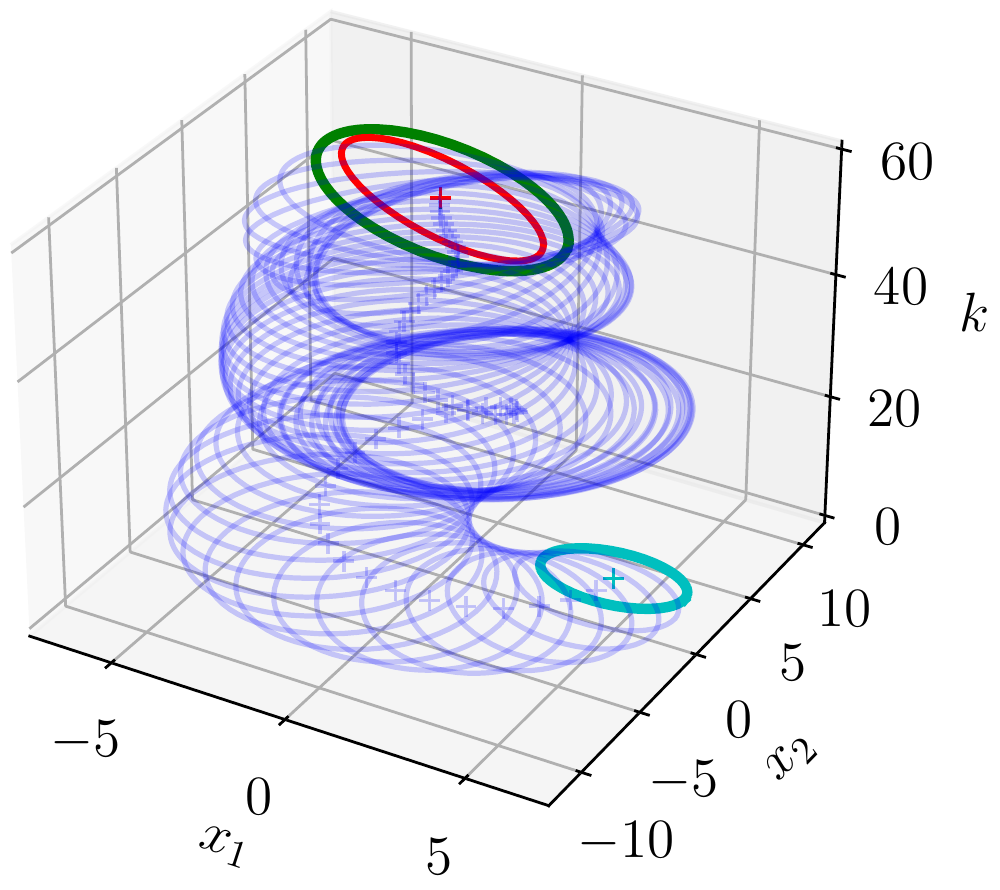}
\caption{$\lambda=2.0$}\label{subfig:state-evolution-lambda20}
\end{subfigure}
\begin{subfigure}{0.32\linewidth}
\includegraphics[height=\linewidth,width=\linewidth]{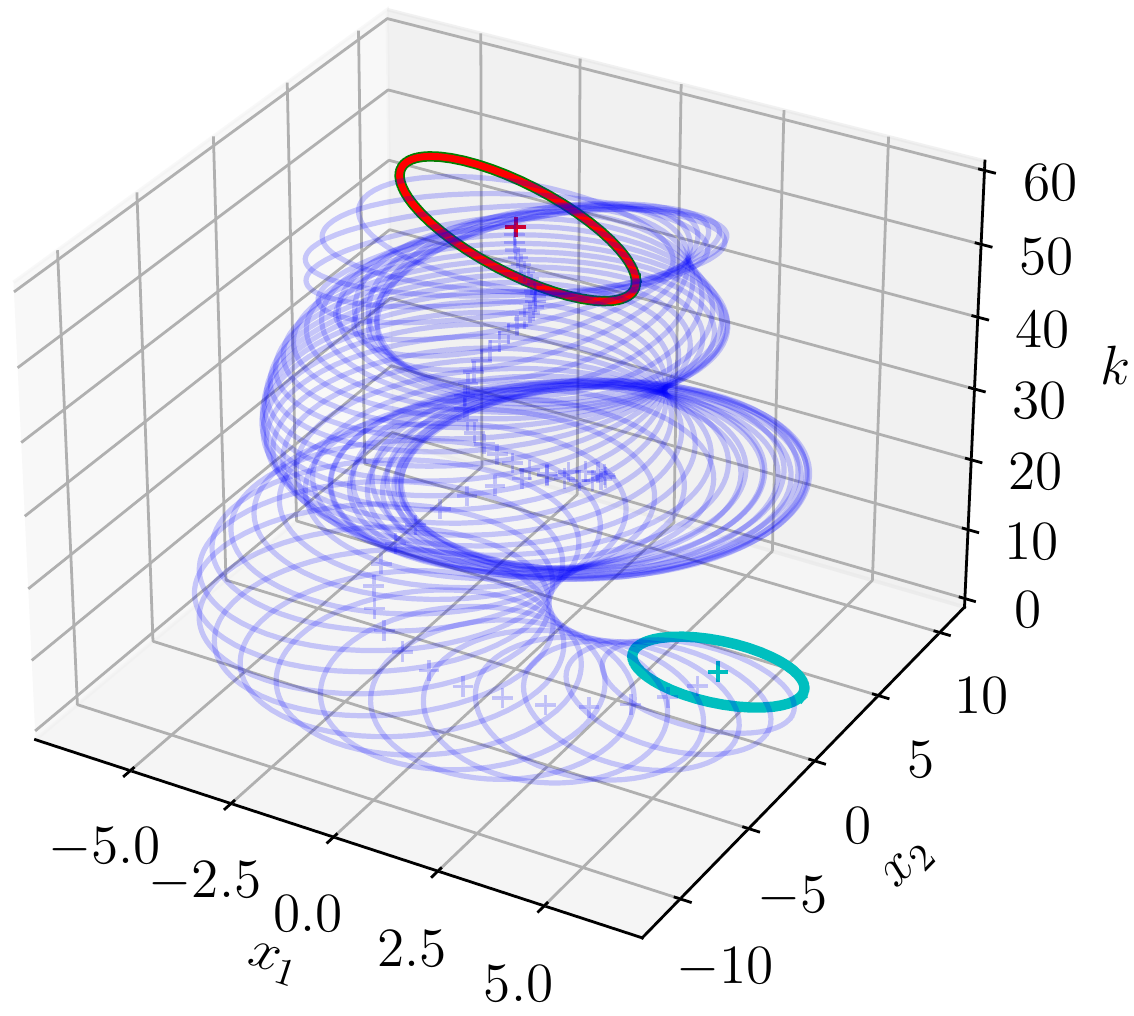}
\caption{$\lambda \rightarrow \infty$}\label{subfig:state-evolution-lambdainfty}
\end{subfigure}
\caption{
Evolution of state statistics illustrated in terms of the 2-$\sigma$ confidence ellipsoids. Cyan, green and red ellipsoids show the initial, terminal and desired distributions, respectively. Figures \ref{subfig:state-evolution-lambda05},  \ref{subfig:state-evolution-lambda20} and \ref{subfig:state-evolution-lambdainfty} show the evolution of the state statistics when $\lambda=0.5$, $\lambda=2.0$, and $\lambda \rightarrow \infty$, respectively. 
}\label{fig:3d-state-evolution}
\end{figure*}

Fig. \ref{fig:3d-state-evolution} illustrates the evolution of the state statistics under the optimal policy found by solving the SDP in \eqref{problem:sdp-formulation}. 
As it can be seen from Figures \ref{subfig:state-evolution-lambda05}-\ref{subfig:state-evolution-lambdainfty}, as $\lambda$ increases the terminal distribution approaches the desired distribution. 

\begin{figure*}[h]
\centering
\begin{subfigure}{0.48\linewidth}
\begin{tikzpicture}
\begin{axis}[
        height=0.55\linewidth,
        width=\linewidth,
        xlabel={$N$},
        ylabel= Computation Time (s),
        ylabel style={at={(0.05, 0.5)}},
        grid=both,
        grid style={line width=0.01pt},
        legend style={nodes={scale=0.9},at={(0.2,0.9)},anchor=north}
    ]
    \addplot table [x=N, y=disturbance-simple, col sep=comma] {computation_converter_all.csv};
    \addlegendentry{Dist-Simple}
    \addplot table [x=N, y=auxilary-simple, col sep=comma] {computation_converter_all.csv};
    \addlegendentry{SH-Simple}
    \addplot table [x=N, y=state-fb, col sep=comma] {computation_converter_all.csv};
    \addlegendentry{SDP}
\end{axis}
\end{tikzpicture}
\caption{}\label{subfig:computation-simple}
\end{subfigure}
\begin{subfigure}{0.48\linewidth}
\begin{tikzpicture}
\begin{axis}[
        height=0.55\linewidth,
        width=\linewidth,
        xlabel={$N$},
        ylabel= Computation Time (s),
        ylabel style={at={(0.05, 0.5)}},
        ylabel style={at={(axis description cs:0.03,.5)}},
        grid=both,
        grid style={line width=0.01pt},
        legend style={nodes={scale=0.9},at={(0.2,0.9)},anchor=north}
    ]
    \addplot table [x=N, y=disturbance-full, col sep=comma] {computation_converter_all.csv};
    \addlegendentry{Dist-Full}
    \addplot table [x=N, y=auxilary-full, col sep=comma] {computation_converter_all.csv};
    \addlegendentry{SH-Full}
    \addplot table [x=N, y=state-fb, col sep=comma] {computation_converter_all.csv};
    \addlegendentry{SDP}
\end{axis}
\end{tikzpicture}
\caption{}\label{subfig:computation-full}
\end{subfigure}
\begin{subfigure}{0.48\linewidth}
\begin{tikzpicture}
\begin{axis}[
        height=0.55\linewidth,
        width=\linewidth,
        xlabel={$N$},
        ylabel= Optimal Values,
        ylabel style={at={(0.05, 0.5)}},
        grid=both,
        grid style={line width=0.01pt},
        legend style={nodes={scale=0.9},at={(0.2,0.9)},anchor=north}
    ]
    \addplot table [x=N, y=disturbance-simple, col sep=comma] {optimal_converter_all.csv};
    \addlegendentry{Dist-Simple}
    \addplot table [x=N, y=auxilary-simple, col sep=comma] {optimal_converter_all.csv};
    \addlegendentry{SH-Simple}
    \addplot table [x=N, y=state-fb, col sep=comma] {optimal_converter_all.csv};
    \addlegendentry{SDP}
\end{axis}
\end{tikzpicture}
\caption{}\label{subfig:optimal-simple}
\end{subfigure}
\begin{subfigure}{0.48\linewidth}
\begin{tikzpicture}
\begin{axis}[
        height=0.55\linewidth,
        width=\linewidth,
        xlabel={$N$},
        ylabel= Optimal Values,
        ylabel style={at={(0.05, 0.5)}},
        ylabel style={at={(axis description cs:0.03,.5)}},
        grid=both,
        grid style={line width=0.01pt},
        legend style={nodes={scale=0.9},at={(0.2,0.9)},anchor=north}
    ]
    \addplot table [x=N, y=disturbance-full, col sep=comma] {optimal_converter_all.csv};
    \addlegendentry{Dist-Full}
    \addplot table [x=N, y=auxilary-full, col sep=comma] {optimal_converter_all.csv};
    \addlegendentry{SH-Full}
    \addplot table  [x=N, y=state-fb, col sep=comma]{optimal_converter_all.csv};
    \addlegendentry{SDP}
\end{axis}
\end{tikzpicture}
\caption{}\label{subfig:optimal-full}
\end{subfigure}
\caption{Computation-time (\ref{subfig:computation-simple}, \ref{subfig:computation-full}) and Optimal Values (\ref{subfig:optimal-simple}, \ref{subfig:optimal-full}) for increasing $N$. Dist and SH represent the affine disturbance feedback policy and the state history feedback policy, respectively. On Figures \ref{subfig:computation-simple} and \ref{subfig:optimal-simple} only terms from the block diagonal or the last disturbance term are used to parametrize the feedback control policy whereas on Figures \ref{subfig:computation-full} and \ref{subfig:optimal-full}, the full state or disturbance history is used. }\label{fig:comparison-LMIwithPrevious}
\end{figure*}
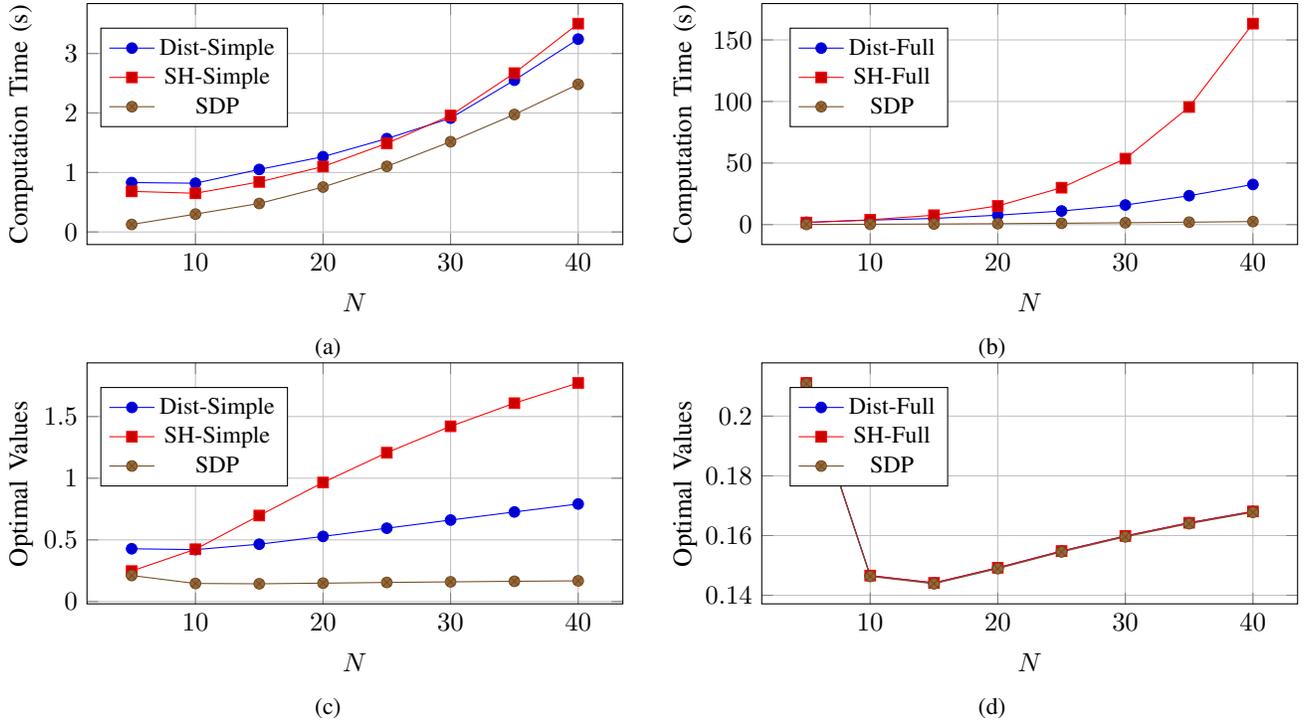

Figures \ref{subfig:computation-simple} and \ref{subfig:optimal-simple} show comparison results between the randomized state feedback policy in \eqref{eq:random-state-policy}, the truncated affine disturbance feedback policy \cite{p:balci2021covariancedisturbance} and the truncated state history feedback policy \cite{p:balci2020covariancewasserstein}. 
In the utilized truncated affine disturbance feedback policy, only the last disturbance term $w_{k-1}$ from the history of disturbances is used for the computation of the control input $u_k$ at time step $k$
(in contrast with the standard or non-truncated affine disturbance feedback policy in which the whole history of disturbances $\{w_{0}, \dots, w_{k-1}\}$ is used instead). The state history feedback policy is parametrized by the matrix variable $\bmtheta$ which is defined in \eqref{eq:variable-transformation1} and which is constrained to be a lower block triangular matrix to preserve the causality of the latter control policy. In the utilized truncated state history feedback policy, the matrix $\bmtheta$ is further restricted to be a block diagonal matrix. 
Note that this truncation scheme does not necessarily mean that the control will be determined by a truncated history of states since
the matrix variable $\cK$ in \eqref{eq:variable-transformation1-eq2} is still lower block triangular (and not necessarily block diagonal) even if $\bmtheta$ is constrained to be a block diagonal matrix. 
However, the number of decision variables that appear in the truncated state history feedback policy for the CS problem is smaller than the standard (non-truncated) version of this policy which typically leads to smaller computational cost.
For both of the truncated policies used in our comparisons, the computation time increases with the same rate with increasing problem horizon $N$ as in the (memoryless) state feedback policy induced by the SDP formulation in \eqref{problem:sdp-formulation}. However, the optimal parameters for the truncated policies yield sub-optimal results compared with our proposed SDP-based (memoryless) state feedback policy as shown in Fig. \ref{subfig:optimal-simple}.
If one uses the standard (or non-truncated) versions of the disturbance affine feedback or state history feedback control policies, the sub-optimality gap will be eliminated, as shown in Fig. \ref{subfig:optimal-full}, however the CS problem becomes computationally intractable as the problem horizon $N$ increases, as shown in Fig. \ref{subfig:computation-full}.

In contrast to the SDP formulation, the other policy parametrizations lead to non-convex problem formulations. 
However, since the values of different local minima are close to each other in this problem instance, optimal values returned by these approaches turned out to be close to each other in Fig. \ref{subfig:optimal-full}.

\begin{table*}[h]
    \centering
    \begin{tabular}{|c||c|c|c|c|c|c|c|c|c|c|c|c|c|}
    \hline
        N & 30 & 40 & 50 & 60 & 70 & 80 & 90 & 100 & 110 & 120 & 130 & 140 & 150  \\
        \hline
        System in Fig. \ref{fig:3d-state-evolution} & 1.45 & 2.40 & 3.60 & 5.22 & 6.77 & 8.77 & 10.88 & 13.38 & 16.11 & 19.05 & 22.24 & 25.75 & 29.34 \\
        \hline
        Double-Int.(2d) & 3.62 & 6.20 & 9.55 & 13.48 & 18.48 & 23.99 & 29.85 & 37.44 & 44.04 & 52.39 & 61.27 & 70.34 & 80.37 \\
        \hline
        Double-Int.(3d) & 6.81 & 11.85 & 18.44 & 26.23 & 35.30 & 45.45 & 57.28 & 71.04 & 85.38 & 100.31 & 117.52 & 136.73 & 156.31 \\
        \hline
    \end{tabular}
    \caption{Computation time of different problem instances in seconds required for the solution of the associated SDP in \eqref{problem:sdp-formulation}. The first row indicates the problem horizon $N$, the second, the third and the fourth row correspond to the randomly generated system in Fig. \ref{fig:3d-state-evolution}, the two-dimensional and the three-dimensional double integrator, respectively.}
    \label{tab:comp_LMI}
\end{table*}

The time complexity of SDP problems is in $\mathcal{O}(\ell^3)$ where $\ell$ is the number of decision variables \cite{p:helmberg1996complexitySDP}. Let $n, m$ be the dimension of the state $x_k$ and control input $u_k$, respectively. Then, the number of decision variables corresponding to the covariance problem is given by $N(n^2 + nm + m^2) + n^2$ and thus, the number of decision variables is linearly proportional to $N$. Therefore, the time complexity of solving Problem \ref{problem:first formulation} over randomized state feedback policies is in $\mathcal{O}(N^3)$.
This fact is also reflected by the computation times reported in Table \ref{tab:comp_LMI}.

\section{Conclusion}\label{s:conclusion}
In this paper, we addressed a class of covariance steering problems for discrete-time stochastic linear systems with Wasserstein terminal cost. 
We showed that the proposed state feedback policies have significant advantages over previous policy parametrizations and formulated an SDP to find the optimal policy parameters by using suitable variable transformations.
Then, we showed that the optimal policy is deterministic for the problems that we addressed. 
Finally, we demonstrate the efficacy of our approach by extensive numerical simulations.

In our future work, we plan to extend our approach to robust and nonlinear covariance control problems as well as dynamic games with covariance assignment constraints.
\bibliographystyle{ieeetr}
\bibliography{bibfile}

\end{document}